\documentclass[a4paper,11pt]{amsart}

\usepackage[top=1.65in, bottom=1.65in, right=1.3in, left=1.3in]{geometry}

\usepackage{color}
\usepackage{xcolor}

\usepackage [utf8]{inputenc}
\usepackage[all]{xy}
\usepackage[english]{babel}
\usepackage{amssymb}
\usepackage[lite, initials]{amsrefs}

\usepackage{enumitem}

\newtheorem{teorema}{Theorem}[section]
\newtheorem*{theorem*}{Structure Theorem}
\newtheorem{lemma}[teorema]{Lemma}
\newtheorem{propos}[teorema]{Proposition}
\newtheorem{corol}[teorema]{Corollary}
\theoremstyle{definition}
\newtheorem{ex}[teorema]{Example}
\newtheorem{rem}[teorema]{Remark}
\newtheorem{defin}[teorema]{Definition}

\def\R{{\mathbb R}}

\def\C{{\mathbb C}}

\newcommand{\Ci}{\mathcal{C}}

\DeclareMathOperator{\spt}{spt}

\def\oli{\overline}				

\def\de{\partial}
\def\debar{\oli{\de}}

\title[Generalized Levi currents and singular loci]{Generalized Levi currents and
singular loci for families
of plurisubharmonic functions}
\author[F.~Bianchi]{Fabrizio Bianchi\textsuperscript{1}}
\address{\textsuperscript{1}CNRS,  Univ.\ Lille, UMR 8524 - Laboratoire Paul Painlev\'e, F-59000 Lille, France}
\email{fabrizio.bianchi@univ-lille.fr}
\author[S.~Mongodi]{Samuele Mongodi\textsuperscript{2}}
\address{\textsuperscript{2}Dipartimento di Matematica e Applicazioni, Universit\`a di Milano Bicocca, Via Roberto Cozzi 55,
I--20125 Milano, Italy.}
\email{samuele.mongodi@unimib.it}

  \date{\today}

\begin{document}
\begin{abstract}
We show how the formalism of Levi currents on complex manifolds,
as introduced by Sibony,
can be used to study the analytic structure of singular sets
associated to families of plurisubharmonic functions, in the sense of Slodkowski.
\end{abstract}

\maketitle

\section{Introduction}

The theory of several complex variables originated as
the study of holomorphic functions; however, soon enough,
plurisubharmonic functions made their appearence and proved themselves
as a useful instrument to solve problems originated in the holomorphic category.

Probably, one of the best known and oldest examples of this phenomenon
is the Levi problem (\cite{Levi}, \cite{Siu} for a survey): characterizing domains
of holomorphy in terms of the pseudoconvexity of the boundary; Oka's solution
\cites{Oka1,Oka2} for domains in $\C^n$ highlighted the role of strictly
plurisubharmonic exhaustions, whose existence was later shown to be
equivalent to Steinness, by Grauert for manifolds \cite{Gra} and Narasimhan
for analytic spaces \cites{Na1,Na2}.

The geometric counterpart of holomorphic functions is represented by complex
analytic varieties, which are locally given as zeroes of holomorphic functions.
This notion does not have a straightforward analogue for plurisubharmonic functions:
they lack the rigidity of holomorphic functions, hence the geometry they describe with
their level sets is not significant, in relation with complex analytic varieties; in fact,
quite the opposite is true: the level sets of a strictly plurisubharmonic function
do not support any kind of complex structure and are examples of B-regular sets
(i.e., sets where the (restrictions of) plurisubharmonic functions are dense in the
continuous ones, introduced in \cites{Sib_classe,Sib_coll}),
which play an important role in the study of the regularity of the $\debar$-problem (see \cite{Sib_classe}).

\medskip

However, there is a rigidity property shared between holomorphic
and plurisubharmonic functions: holomorphic functions obey a maximum
modulus property, which, for instance, characterizes algebras of holomorphic
functions on plane domains (see Rudin \cite{Rudin}),
and (pluri)subharmonic functions likewise satisfy a maximum property.
Indeed, Rudin's theorem can be deduced as a special case of subharmonicity,
in turn obtained via the maximum property \cite{Wer}.

It is quite clear that the maximum property for plurisubharmonic functions does
not hold on arbitrary sets, for example it does not hold on B-regular sets; on the
other hand, given an open domain of an analytic variety, plurisubharmonic functions
attain their maximum on the boundary. It is therefore reasonable to expect that sets
where the maximum property holds should bear some resemblance of complex structure;
this idea originated, more or less explicitly, a number of constructions related to function
algebras, such as Shilov boundary, Jensen boundary, peak points, Choquet boundaries
(see \cite{Gam} for a comprehensive exposition), and some of these were also employed
to give more general definitions of plurisubharmonicity in the context of uniform algebras \cite{GS}.

By localizing the idea of Jensen boundary, we obtain the notion of local maximum
set for $h$-plurisubharmonic functions \cite{Sl}; these sets enjoy many properties
of analytic varieties, with $h+1$ playing the role of the dimension. Moreover, local
maximum sets for plurisubharmonic functions (usually just called local maximum
sets) are $1$-pseudoconcave, in the sense that their complement is $(n-2)$-pseudoconvex
according to Rothstein \cite{Roth}, a property which is true also for complex analytic varieties
of dimension at least $1$. So, for instance, the largest local maximum set contained in the
boundary of a pseudoconvex set will contain any germ of analytic variety and, more
generally, any positive, $\de\debar$-closed current 
of bidimension $(1,1)$ and with compact support
\cite{OS}.

\medskip

In particular, compact local maximum sets, just like compact analytic varieties,
force plurisubharmonic functions (and hence holomorphic functions) to be constant
on them; therefore, no strictly plurisubharmonic function can exist in a neighbourhood
of a compact local maximum set. This consideration was the core of the investigation of
weakly complete spaces, started by Slodkowski and Tomassini in \cite{ST},
with the definition of the kernel of a weakly complete space (which is the set of points
where no plurisubharmonic exhaution can be strictly plurisubharmonic) and continued
by them and the second author in a series of papers \cites{M,MZ,mst,crass,mst2,MT}.

At the same time, Sibony introduced and studied (\cites{Sib_pf,Sib,Sib_pseudo})
the idea of Levi current: a positive current of bidimension $(1,1)$
which is $\de\debar$-closed
and vanishes when wedged with $\de\debar u$, for $u$ a plurisubharmonic function;
in a previous paper \cite{BM}, we investigated the relation between Levi currents,
local maximum sets and the kernel of a weakly complete space \cite{ST}. It is quite
clear that local maximum sets are, as the name suggests, just sets, whereas currents
imply more structure; in fact, while it is clear that the support of a Levi current is a
local maximum set and while it is possible, given a local maximum set, to produce
a Levi current with support contained in it, we do not know 
yet whether it is true that every local maximum set \emph{is} the support of a Levi current.

It is interesting to note that, already in \cite[Remark 1.5 - (i)]{BrSib},
 the authors noted the link between local maximum sets and pluriharmonic positive currents.

Again in the work
\cite{Sib}, Sibony also briefly mentioned Liouville currents, whose
definition is similar to the one of Levi current, but in the last property: one asks
that $T\wedge \de\debar u=0$ for every $u$ plurisubharmonic and bounded;
as Levi currents are related to the kernel of a weakly complete space,  Liouville currents
could be linked to the core of a complex space, as defined and investigated by
Harz, Shcherbina, and Tomassini in \cites{HST1,HST2,HST3},
which is the set of points where no bounded plurisubharmonic function can be
strictly plurisubharmonic.

In a recent paper
\cite{Sl_pseudo}, Slodkowski showed that the core is a union of ``primitive" sets,
which are 1-pseudoconcave and enjoy a Liouville property, namely every bounded
plurisubharmonic function is constant on them; this result had been previously
obtained for complex surfaces in \cite{HST2} and was also independently
proved in the general case by Poletski and Shcherbina in \cite{PS}. Slodkowski proved
this by tackling a more general problem, with respect to a family of
plurisubharmonic functions, satisfying some given properties, called admissible class.

\medskip

In the present paper, we intend to deepen the investigation of the
relations between plurisubharmonic functions (and local maximum sets) and
positive currents of bidimension $(1,1)$ (as generalizations of complex analytic varieties or Levi-flat sets),
by expanding the results of \cite{BM}: 
we introduce a generalization of Levi currents
(the $\mathcal F$-currents, see Definition \ref{def-Levic}),
where the condition $T\wedge \de\debar u $ should hold for $u$ in
an \emph{admissible class} $\mathcal F$, 
as defined by Slodkowski in \cite{Sl_pseudo}, see Definition \ref{defi-admissible-class}.
To any such class, we can also associate a \emph{singular locus},
 as the set where no element
of the class can be strictly plurisubharmonic.
Examples of admissible classes, and of the corresponding singular sets,
are given by all plurisubharmonic function (and the \emph{minimal kernels},
as introduced in \cite{ST}),
or by all bounded plurisubharmonic functions (and the \emph{cores}, as introduced in \cite{HST1}).

The following is our main result.

\begin{teorema}\label{teo-main-singular-currents}
Let $X$ be a complex manifold
and let $\mathcal F$ be an admissible class.
\begin{enumerate}
\item All $\mathcal F$-currents are supported in the singular locus of $\mathcal F$.
If  $\mathcal F$
 contains an exhaustion function,
  the singular locus is empty if and only if there are no $\mathcal F$-currents.
  \item 
  All elements of $\mathcal F$
are constant on the supports of extremal $\mathcal F$-currents.
\item There exists an $\mathcal F$-current whose support 
 is equal to the union of 
 the supports of all $\mathcal F$-currents. 
  \item Assume that $\mathcal F(X)$
contains an exhaustion function. Let $T$ be a $\mathcal F$-current.
If
$\spt T$ is compact, then
it is a local maximum set.
\item Assume that $K \subset X$ is an $\mathcal F$-component,
or
a compact local maximum set.
 Then there exists
$T \in \hat {\mathcal F}$
such that $\spt T \subseteq K$.
\end{enumerate}
\end{teorema}

The paper is organized as follows. In 
Section \ref{s:admissible} we define admissible classes and their
associated singular loci,
and we state the properties that we need in the sequel. In Section \ref{s:currents}
we introduce the notion of $\mathcal F$-currents, and study such currents and their supports.
The proof of Theorem \ref{teo-main-singular-currents} is then given in Section \ref{s:proof-thm}.
We conclude the paper with some remarks about the decomposition of $\mathcal F$
on the levels sets of elements of $\mathcal F$, and on the localization of the definitions in the paper to the case of compact subsets,
see Section \ref{s:final-remarks}.

\bigskip

\paragraph{\bf Acknowledgements}
We are indebted to Nessim Sibony and Zbigniew Slodkowski
for discussions about our previous work \cite{BM}, which lead to the current paper.

This project has received funding from
 the French government through the Programme Investissement d'Avenir
 (I-SITE ULNE /ANR-16-IDEX-0004,
 LabEx CEMPI /ANR-11-LABX-0007-01,
ANR QuaSiDy /ANR-21-CE40-0016,
ANR PADAWAN ANR-21-CE40-0012-01)
managed by the Agence Nationale de la Recherche.

\section{Admissible families and singular loci}
\label{s:admissible}

In this section, following \cite{Sl_pseudo},
we define admissible class
and the associated singular loci, 
 and we recall their properties that we will need in the sequel.
We fix a complex manifold $X$.
We say that an open subset $U \subset X$ is \emph{allowable} if it is either
relatively compact in $X$ or cocompact. Notice in particular that  $X$ is allowable.

\begin{defin}\label{defi-admissible-class}
An \emph{admissible class} $\mathcal F$
is the datum, for every allowable open set $U \subset X$, of a family $\mathcal F (U)$
of continuous plurisubharmonic functions satisfying the following properties
for every allowable sets $W,U,U_1, \dots U_m$:
	\begin{enumerate}[label={\bf (A\arabic*)}]
\item for every sequence $\phi_n\in \mathcal F(X)$, there exists a sequence 
of positive
$\epsilon_n \in \mathbb R$ such that
the series
 $\sum_{n=1}^\infty \epsilon_n \phi_n$ converges locally uniformly to an element of $\mathcal F (X)$;
\item whenever $\phi \in \mathcal F (U)$ and $W \subset U$,
 then $\phi_{|W} \in \mathcal F (W)$;
\item\label{def-glue} if $\{U_1, \dots U_m\}$ is a finite cover of $X$
 and $\phi \colon X\to\mathbb R$ is such that $\phi_{|U_i}\in \mathcal F(U_i)$
 for all $1\leq i \leq m$, then $\phi \in \mathcal F(X)$;
\item\label{def-cone-bounded}
  the set $\mathcal F (U)$ is a convex cone and contains all bounded
 $\Ci^\infty$ plurisubharmonic functions on $U$;
\item\label{def-convex}
  for every 
 $\phi_1, \dots \phi_m\in \mathcal F (U)$, and every 
 $\Ci^\infty$ convex function 
 $v\colon \mathbb R^m \to \mathbb R$
 of at most linear 
 growth 
 and such that $\frac{\partial v}{\partial t_i}\geq 0, i=1, \dots, m$
 on the joint range
 of $(\phi_m, \dots, \phi_m)$,
  the
 function 
 $\phi:= v(\phi_1, \dots, \phi_m)$
 belongs to $\mathcal F(U)$;
\item\label{def-perturb-t}
 for every $\phi \in \mathcal F(U)$
 which is strongly plurisubharmonic on $U$
 and $\rho \in \Ci^\infty (U)$
 with $\overline {\spt \rho} \subset U$, there exists $t>0$ such that
 $\phi + t \rho \in \mathcal F(U)$.   
\end{enumerate}
\end{defin}

\begin{rem}\label{rmk:a7}
Observe that condition \ref{def-cone-bounded} implies the following property:
\begin{enumerate}[label={\bf (A\arabic*)}]
\setcounter{enumi}{6}
\item\label{def-new-dense}
every point $p\in X$ admits
an open neighbourhood $\Omega$
  such that, for all allowable
  open subset
   $V\subset \Omega$, 
  $\mathcal F (V)\cap \Ci^\infty$ is dense
(for the topology of locally uniform convergence)  
   in $\mathcal F(V)$
   \end{enumerate}
  We will use this property a number of times in the following,
  hence we prefer to add it to the list of properties of an admissible class.
   \end{rem}

\begin{rem}
In general, an allowable class $\mathcal F$ is 
not
closed by maximum, i.e., the maximum of two elements in
a given class $\mathcal F$
is not necessarily in $\mathcal F$. On the other hand,
by the condition \ref{def-convex}, the class is closed by
$\max_\epsilon$ for all positive $\epsilon$, where
$\max_{\epsilon} (x,y)$ is any smooth approximation of the function $\max (x,y)$.
\end{rem}

\begin{defin}
The \emph{singular locus} of 
an admissible
 class $\mathcal F$ is the complement $\Sigma^{\mathcal F} = \Sigma^{\mathcal F}_X$ of the 
 set of points $x\in X$ such that there exists $\phi \in \mathcal F(X)$
 which is strongly plurisubharmonic at $x$.
\end{defin}

\begin{defin}
An $\mathcal F$-component is an equivalence class of the relation $\sim$ defined as follows:
$x\sim y$ if $\phi(x) = \phi(y)$ for all $\phi \in \mathcal F$. 
\end{defin}

Observe that this relation is meaningful mostly on $\Sigma^{\mathcal F}$. Indeed,
the component of a point outside of $\Sigma^{\mathcal F}$ is given by that single point.

\medskip

The following are examples of 
allowable families, giving rise to well-studied
singular sets,
 see \cite{Sl_pseudo}*{Section 5} for more details.

\begin{ex}\label{ex:kernel}
Let us consider, for a given $k \in\mathbb N \cup\{\infty\}$,
 the class $\mathcal F$ defined as follows:
\begin{itemize}
\item 
the set of all lower-bounded, $\Ci^k$ psh functions on $U$, for all relatively compact open set $U$, and
\item the set of all lower-bounded, $\Ci^k$ psh functions $\phi$ on $U$, such that the sub-levels sets
$\{\phi \leq c\}$ are relatively compact in $X$ for all $c \in \mathbb R$,
for all cocompact open set $U$.
\end{itemize}

The singular loci associated to $\mathcal F$ as above are the \emph{minimal kernels},
as introduced and studied in \cites{ST,mst}; in particular, in the series of paper \cites{mst, crass, mst2, MZ, M, MT}, the authors considered the case of complex surfaces such that $\mathcal{F}(X)$ contains a real analytic exhaustion. As a consequence of this detailed study, one notices that, in such a case, the singular locus does not depend on the regularity $k$.
\end{ex}

\begin{ex}\label{ex:core}
Let us now consider, for a given $k \in\mathbb N\cup\{\infty\}$ 
 the class $\mathcal F$ given, on any admissible open set $U$, 
 by all uniformly bounded $\Ci^k$ plurisubharmonic functions.

The singular loci in this case correspond to the \emph{cores}, as introduced and studied 
in \cites{HST1,HST2,HST3}. It is known that, in the case of cores, regularity plays a role (see \cite{H}).
\end{ex}

In the remaining part of this section, we summarize the
main properties of singular loci, mainly from \cite{Sl_pseudo},
 that we will need in the sequel. We first need to recall the following further definitions.

\begin{defin}
Given an admissible class $\mathcal F$, an
 element $\phi \in \mathcal F(X)$ is a $\mathcal F$-\emph{minimal function}
 if it is strongly plurisubharmonic on $X \setminus \Sigma_X^{\mathcal F}$.
\end{defin}

\begin{defin}\label{def:locmax}
Let $Z$ be a locally closed set. We say that  $Z$ is a local maximum set if every $x\in Z$ 
admits a neighbourhood $V$  with the following property: for every compact set $K\subset V$
and every function $\phi$ which is psh in a neighbourhood of $K$,  we have
\begin{equation}\label{eq:locmax}
\max_{Z \cap K}
\psi
  = \max_{Z\cap bK} \psi.
  \end{equation}
\end{defin}

\begin{rem}\label{rem-local-max-set-equiv}
More generally, given ad admissible class $\mathcal F$
as in Definition \ref{defi-admissible-class}, one may
say that $Z$ is a local maximum set for 
   $\mathcal F$ if the condition in Definition
   \ref{def:locmax} is satisfied for all $\phi \in \mathcal F (V)$, where
$V$  is an allowable open neighbourhood of $Z$.
  However, it is not difficult to see that this is actually equivalent
  to be a local maximum set.
 It is clear that any local maximum set is a local maximum set for $\mathcal F$. On the other hand, take
  let $Z$ be a  local maximum set  for $\mathcal F$.
 Take any point of $Z$, an open neighbourhood $U$
   and a compact set $K\subset U$.
  We can, without loss of generality, reduce $U$ and assume that it is relatively compact, hence admissible.
Take a psh function $\phi$
on a neighbourhood of $K$. Again, without loss of generality, we can assume that
$\phi$ is defined on $U$. If $\phi$ is smooth, 
  by condition \ref{def-cone-bounded} in Definition \ref{defi-admissible-class}, $\phi$ belongs to 
$\mathcal F (U)$. Hence \eqref{eq:locmax} holds for $\phi$. The statement for a general upper semicontinuous plurisubharmonic
$\phi$ now follows
by
approximation.
\end{rem}

\begin{teorema}[Theorem 4.2 of \cite{Sl}]
Let $X$ be a $n$-dimensional complex manifold.
 A closed set $Z$ is a local maximum set if and only if it is \emph{1-pseudoconcave}: it can be covered
 by open sets $V_i$ such that $V_i \setminus Z$ admits a  $(n-2)$-plurisubharmonic exhaustion function.

\end{teorema}

Recall that, for a $\mathcal{C}^2$ function, to be $(n-2)$-plurisubharmonic means
that its complex Hessian has at least $2$ non-negative eigenvalues.

\medskip

The next result gives a
 characterization of local maximum sets
 in terms of the local behaviour of admissibile functions.
 Although, by Remark \ref{rem-local-max-set-equiv}, such result is implied
by \cite[Proposition 2.3]{Sl},
we give here a proof of this, to show how the definition of
admissible classes
precisely allows one to work as if doing so 
in the
algebra of psh functions.

\begin{propos}\label{p:equivalence-locmax}
Let $Z$ be a locally closed set and $\mathcal F$
be an admissible class.
The following conditions are equivalent.
\begin{enumerate}
\item $Z$ is a local maximum set for the class $\mathcal F$;
\item there do not exist $z^*\in Z$, $r>0$, $\epsilon >0$ and a strictly psh function $u$
in $\mathcal F (B(z^*,r))$ such that $u(z^*)=0$ and $u(z)\leq -\epsilon |z-z^*|^2$ for $z \in Z \cap B(z^*,r)$.
\end{enumerate}
\end{propos}

\begin{proof}
First, it is clear that the existence of a function $u$
as in the second item contradicts the fact that $Z$ is a local maximum set for the class $\mathcal F$. 

For the other implication, suppose
that $Z$ is not a local maximum set for the class $\mathcal F$. Then there exists a compact
subset $K \subset Z$, an allowable open
neighbourhood $U$
of $K$, and an element $u_0 \in \mathcal F (U)$
such that
\[
\max_{K \cap Z} u_0 > \max_{\partial K \cap Z} u_0.
\]
By the density of smooth elements in $\mathcal F (U)$,
 we can assume that $u_0 \in \Ci^2(U)$.
Then,
\cite[Lemma 2.2]{Sl}
gives the existence
of a strictly convex function $f\colon U \to \mathbb R$
and a point $x^* \in K \setminus \partial K$
such that 
\[
(u_0+f)(x^*)=0 \quad \mbox{ and } \quad
(u_0+f)(x)\leq -\epsilon |x-x^*|^2 \mbox{ for } x\in K. 
\]
(the lemma is stated in $\mathbb C^n$ --
and actually $\mathbb R^n$, just for upper semicontinuous functions --
but the construction is local). By taking $r$ sufficiently small, the function $u:= u_0+f$
is strictly psh on $B(x^*,r)$ and satisfies the requirements in the statement.
\end{proof}

\begin{teorema}[\cite{Sl_pseudo}]\label{teo-summary-sl}
Let $\mathcal F$ be an admissible
class.
\begin{enumerate}
\item there exists a minimal function in $\mathcal F$;
\item the singular locus $\Sigma^{\mathcal F}$ is a local maximum set (hence 1-pseudoconcave), or empty;
\item all $\mathcal F$-components of points in $\Sigma^{\mathcal F}$ are 1-pseudoconcave;
\item if $x\not\in\Sigma^{\mathcal F}$, then the $\mathcal{F}$-component containing $x$ is $\{x\}$.
\end{enumerate}
\end{teorema}

\begin{rem}
Items (1) and (2) in Theorem \ref{teo-summary-sl}
were proved in \cite{ST} in the case of minimal kernels (see Example \ref{ex:kernel})
and in \cite{HST1} in the case of cores (see Example \ref{ex:core}); the decomposition
in $\mathcal F$-components for cores was already proved in the $2$-dimensional case
in \cite{HST2} and extended to every dimension by Poletski and Shcherbina in \cite{PS}.
\end{rem}

\section{Generalized Levi currents}
\label{s:currents}

In this section we fix a complex manifold $X$.
 Following the definition of Levi currents by Sibony \cite{Sib},
 we define a natural generalization of this notion adapted to any admissible class as in Definition
 \ref{defi-admissible-class}.

\begin{defin}\label{def-Levic}
Let $\mathcal F$ be an admissible class. An
$\mathcal F$-current
is a current $T$
on $X$ satisfying the following properties:
\begin{enumerate}[label={\bf (C\arabic*)}]
\item\label{def-Levic-non0} $T$ is non-zero;
\item\label{def-Levic-11} $T$ is of bidimension $(1,1)$;
\item\label{def-Levic-positive} $T$ is positive;
\item\label{def-Levic-ddc} $i \partial \bar \partial T=0$;
\item\label{def-Levic-wedge-ddc} $T \wedge i \partial \bar \partial u =0$ for all $u \in \mathcal F(U)$ for $U$
an allowable neighbourhood of the support of $T$.
\end{enumerate}
We denote by $\hat {\mathcal F}$ the set of all $\mathcal F$-currents.
We say that an $\mathcal F$-current is \emph{extremal}
if $T=T_1=T_2$ whenever $T = (T_1+T_2)/2$ for $T_1, T_2$ $\mathcal F$-currents.
\end{defin}

\begin{ex}
When $\mathcal F$ is as in Example \ref{ex:kernel} (resp.\ Example \ref{ex:core}),
 we recover
  the definition of Levi (resp.\ Liouville) currents, as in \cite{Sib}.
\end{ex}

The following lemma permits to extend the definition
of the  intersections between $\mathcal F$-currents and 
some exact forms. The arguments of the proof 
are given in \cite{Sib}
and are based on a method developed in \cite{DS}.
We use here \ref{def-new-dense} in order to (locally)
approximate continuous elements of $\mathcal F$ with 
smooth ones.

\begin{lemma}\label{lemma-defi-limit}
Let $\mathcal F$ be an admissible class. Take $u, u_n \in \mathcal F$, with
$u_n$ smooth and such that $u_n \to u$, and let $T$
be a
positive closed current on $X$ of bidimension $(1,1)$. Then
the current $T \wedge \partial u$ is well defined
and
\[ 
T \wedge \partial u_n \to T \wedge \partial u.
\]
Similar assertions hold for $T \wedge \bar\partial u, T\wedge \partial u \wedge \bar \partial u$, and
$T \wedge \partial \bar \partial u$.
\end{lemma}

The following properties are a consequence of the previous lemma.
They are stated in \cite[Section 4]{Sib}, see also \cite[Lemma 2.3]{BM},
in the case of Levi currents. Since
a similar proof works also in this more general settings, we will omit it here.

\begin{lemma}\label{lemma-intersections-vanish}
Let $\mathcal F$ be an admissible class. Take $u \in \mathcal F$ and let $T$
be a $\mathcal F$-current. Then
the currents
\[T \wedge \partial u, \quad T \wedge \bar \partial u, \quad \mbox{ and } \quad T \wedge \partial u \wedge \bar \partial u\]
are well defined and vanish identically on $X$.
\end{lemma}

\begin{corol}\label{corol-Tdudu}
Let $\mathcal F$ be an admissible class and take $T \in \hat{ \mathcal F}$.
If $u\in \mathcal F \cap \Ci^1$, 
then 
 the 2-vector field associated to $T$
  belongs to the kernel of 
  $i \partial u \wedge \bar \partial u $
  ($||T||$-almost everywhere), whenever the latter
is non-zero.
\end{corol}

In the statement above, $\|T\|$ denotes the mass measure associated to $T$, see for instance
\cite[p. 310]{F}.

\begin{proof}
The statement is equivalent to $T \wedge i \partial u \wedge \bar \partial u $,
hence follows from Lemma \ref{lemma-intersections-vanish}.
\end{proof}

The next
lemma gives a first indication of the relation between 
the supports of $\mathcal F$-currents and
the points where elements of $\mathcal F$ are strictly psh.
The case of Levi
currents is given in \cite[Corollary 2.6 and Lemma 2.7]{BM}.

\begin{lemma}\label{l:stricly-psh-supports}
Let $\mathcal F$ be an admissible class.
\begin{enumerate}
\item If there exists $u \in \mathcal F$ and $x\in X$ such that
$u$ is strictly psh at $x$, then $x \notin \spt T$ for any $T \in \hat {\mathcal F}$;
\item Assume $T\in \hat {\mathcal F}$
has compact support. If $u$ belongs to $\mathcal F(U)$ 
for some allowable open neighbourhood of $\spt T$,
and is strictly psh at some $x\in U$, then $x \notin \spt T$.
\end{enumerate}
\end{lemma}

\begin{proof}
We prove the two assertion separately.
\begin{enumerate}
\item 
By 
\ref{def-new-dense},
we can assume that $u$ is smooth, and strictly psh near $x$.
Let $U$ be a small neighbourhood of $x$, where
$u$
is strictly psh. Then $U$ is allowable, and $u \in \mathcal F(U)$.
Take a smooth function $\rho$ supported on
$U$.
 By \ref{def-cone-bounded} in Definition \ref{defi-admissible-class},
 we have
that $u + \rho \in \mathcal F (U)$.
If now $T$ is an element of $\hat {\mathcal F}$,
by Lemma \ref{lemma-intersections-vanish}
we must have
$T \wedge i \partial (u + \rho) \wedge \bar \partial (u + \rho) =0$. Since $\rho$
is arbitrary, this implies that $T=0$. Hence, there are no elements of $\hat {\mathcal F}$
having $x$ in their support.
\item As above, by 
\ref{def-new-dense},
we can assume that $u$ is smooth.
Since $u$ is strictly psh at $x$, the same is true in a neighbourhood. In order to prove 
that $x \notin \spt T$, it is then enough to prove that $T \wedge i \partial \bar \partial u =0$ near $x$.
We prove that this is true in the allowable open set $U$.

First observe that $T \wedge i \partial \bar \partial u$ is a positive measure on $U$.
Consider a second open neighbourhood $U'$ of $\spt T$ with $\spt T \subset U' \Subset U$
and let $\chi$ be a smooth function, compactly supported on $U$ and equal to $1$ on $U'$.
By replacing $u$ with $\chi u$, we can assume that $u$ is defined on $X$, psh near $\spt T$, and equal to zero 
near the boundary of $U$.
Since $i \partial \bar \partial T=0$
(by \ref{def-Levic-ddc} in Definition
\ref{def-Levic})
an application of Stokes Theorem gives that
$\langle T \wedge i \partial \bar \partial u,1\rangle=0$. Since 
$T \wedge i \partial \bar \partial u$ is a positive measure, 
this completes the proof.
\end{enumerate}
\end{proof}

A similar application of Stokes Theorem also gives the following result.

\begin{lemma}\label{lemma-14ok}
Let $\mathcal F$ be an admissible class.
Suppose that a current $T$ satisfies the properties
\ref{def-Levic-non0}-\ref{def-Levic-ddc}
in Definition
\ref{def-Levic}.
If 
$T$ has compact support, or if $T$ is supported on a $\mathcal F$-component,
then $T$ satisfies the property \ref{def-Levic-wedge-ddc}.
In particular,
$T \in \hat {\mathcal F}$.
\end{lemma}

\begin{proof}
Assume first that $T$ as compact support.
Take $U$ an allowable relatively compact open neighbourhood of the support of $T$
and fix $u \in \mathcal F (U)$. The current $T \wedge i \partial \bar \partial u$ 
is well defined by Lemma \ref{lemma-defi-limit}
(and conditions \ref{def-Levic-11} and \ref{def-Levic-positive}).
It is also compactly supported in $U$ and,
 since $T$ is positive (by \ref{def-Levic-positive}),
 is a positive measure. We show that
$\langle T \wedge i \partial \bar \partial u , 1\rangle =0$.
Again by Lemma \ref{lemma-defi-limit}
(and  conditions \ref{def-Levic-11} and \ref{def-Levic-positive}),
the currents 
$T \wedge i \partial u$ and $T \wedge i \bar \partial u$ are also well defined. 
By Stokes theorem, we
have that $\langle i \partial \bar \partial (uT), 1\rangle=
\langle i \partial ( \bar \partial u \wedge T) , 1\rangle 
 =  \langle i \bar \partial (\partial u \wedge T), 1\rangle =0$.
The assertion follows (again by Stokes theorem and \ref{def-Levic-ddc}) in this case.

In the case where the support of $T$
 is contained in a $\mathcal F$-component,
 condition \ref{def-Levic-wedge-ddc}
 in Definition \ref{def-Levic}) is trivially
 satisfied since any $u \in \mathcal F$ is constant on the support of $T$. This completes the proof.
\end{proof}

We conclude this section with the following lemma, giving
a relation between the existence of $\mathcal F$-currents and the absence
of strictly psh functions.

\begin{lemma}\label{lemma:claim-prop-empty}
Let $K$ be a compact set, or a $\mathcal F$-component.
 If there are no
$\mathcal F$-currents supported on $K$,
there exists an allowable open neighbourhood $U$
of $K$
and an element of $\mathcal F(U)$
which is strictly psh (on $U$).
 \end{lemma}

 \begin{proof}
 Let $S$ be a positive, $\partial \bar \partial$-closed current
 of bidimension (1,1) supported on $K$.
  By Lemma \ref{lemma-14ok}, $S$ is a $\mathcal F$-current or $S=0$.
  By assumption, we have that $S=0$.
  Hence, we can assume that there are no non-zero positive $\partial \bar \partial$-closed
currents of bidimension (1,1) supported on $K$. 
   To conclude, we 
  use a duality argument as in \cite[Proposition 2.1]{Sib},
   see also \cites{OS,Sul}.
  Consider the topological vector space of currents 
  of bidimension $(1,1)$ with the topology of weak convergence; denote by
  $C$ the set of positive currents of bidimension 
  $(1,1)$ of mass 1 (with respect to some Hermitian metric on $X$) and supported on $K$
  and  by $Y$ the set of $\partial \bar \partial$-closed currents on $X$.
  By the arguments above,
  we have $C \cap Y = \emptyset$. By Hahn-Banach theorem,
  $C$ and $Y$ are separated, i.e., there exist $\delta >0$ and a continuous
  linear functional $L$ such that $Y\subseteq \ker L$ and $L(T)>\delta$ for all $T\in C$.

	By definition, $Y$ is the annihilator of the space of $\partial\bar\partial$-exact $(1,1)$-forms
	with compact support, i.e., $Y=\{i\partial\bar\partial u\ :\ u\in\mathcal{C}^\infty_c(X)\}^\perp$.
	
	As the spaces of test forms are reflexive, we have that the continuous linear functional $L$
	can be represented as $L(T)=\langle T, i \partial \bar \partial u\rangle$ for some $u\in\mathcal{C}^\infty_c(X)$.
	
	The separation condition implies that
	$\langle T, i \partial \bar \partial u\rangle\geq \delta$ for all $T \in C$;
	if we test this condition against  the current $T=\delta_x i \xi\wedge\overline{\xi}$
	with $x \in K$
	(where  $\delta_x$ is the Dirac mass at $x$ and $\xi$ is 
	any $(1,0)$ tangent vector at $x$), we obtain that $u$ is strictly subharmonic on the disc
	through $x$ with complex direction $\xi$. The function $u$ is then strictly psh on a neighbourhood 
 of $K$, hence it belongs to $\mathcal F(U)$ for every allowable neighbourhood of its support, by 
 the property \ref{def-cone-bounded} of admissible classes. 	
The proof is complete.
 \end{proof}

\section{Proof of Theorem \ref{teo-main-singular-currents}}
\label{s:proof-thm}

In this section we prove the assertions in Theorem
\ref{teo-main-singular-currents}.

\begin{propos}\label{prop-support-sigma}
Let $X$ be a complex manifold and $\mathcal F$
 an admissible class.
  All $\mathcal F$-currents are supported in the singular locus $\Sigma^{\mathcal F}$ of $\mathcal F$.
  In particular, if $\Sigma^{\mathcal F}$ is empty, there are no $\mathcal F$-currents.
\end{propos}

\begin{proof}
Take $x \in X$ and assume that there is a function $\phi \in \mathcal F$ which is strictly psh at $x$.
In particular, $\phi \in \mathcal (U)$ for some small allowable neighbourhood of $x$.
 By
  \ref{def-new-dense}
 (and up to
 possibly restricting $U$), 
 we can approximate $\phi$
 by smooth elements $\phi_n$
 of $\mathcal F$, which are then strictly psh on some neighbourhood of $x$.
 We can then assume that $\phi$ is smooth and strictly psh near $x$.
 We then use the property \ref{def-perturb-t}
 in Definition \ref{defi-admissible-class}, applied to
 a family of smooth functions $\rho_i, 1\leq j\leq 2\dim X $
such that the kernels of $\alpha_i := \partial (\phi + t_j \rho_j)\wedge \bar \partial (\phi+ t_j \rho_i)$
are independent over $\mathbb R$, where the $t_j$'s are given by that property. 
This property stays true in a neighbourhood of $x$.

If $T$ is a Levi current whose support contains $x$, by Lemma \ref{lemma-intersections-vanish} we should have
$T\wedge \alpha_i=0$ for all $i$. This gives a contradiction, and concludes the proof.
\end{proof}

\begin{propos}\label{prop-empty}
Let $X$ be a complex manifold
and $\mathcal F$ an admissible class.
Suppose that $\mathcal F$ contains an exhaustion function and
that
 there are no $\mathcal F$-currents. Then there exists 
an element of $\mathcal F(X)$ which is
an exhaustion function and
everywhere
 strictly psh.
 In particular, 
 the singular locus $\Sigma^{\mathcal F}$
is empty.
\end{propos}

\begin{proof}
 We can follow the arguments of the proof
of \cite[Theorem 4.4]{Sib}. We will construct a strictly
psh exhaustion function on $X$ by applying inductively
Lemma \ref{lemma:claim-prop-empty}.

By assumption, $X$ admits an exhaustion function 
$\phi$ in $\mathcal F(X)$.
We can then construct a sequence
of compact sets $K_n$ such that $K_n \Subset \mathring {K_{n+1}}$,
$\cup_n K_n = X$, and with the property that, for all $n$,
\[
K_n = \{x \in X \colon u(x) \leq \max_{y \in K_n} u(u) \quad  \forall u \in \mathcal F\}.
\]
By applying Lemma \ref{lemma:claim-prop-empty},
we can find a sequence of functions $v_n$ which are strictly psh
 on $\mathcal V_n$
on some open allowable neighbourhood $V_n$ of $K_n$.
Choose for every $n$ a convex increasing function $\chi_n \colon \mathbb R \to \mathbb R$
such that
\[
\begin{cases}
\chi_n (\phi) < \inf_{K_n} v_n \mbox{ on } K_{n-1}\\
\chi_n (\phi)  \geq  \sup v_n \mbox{ near } \partial K_n
\end{cases}
\]
and define \[u_n := \widetilde \max(\chi_n (\phi), v_n).\]
where $\widetilde \max$ is some smooth function sufficiently close to $\max$.
This function is then equal
to $v_n$ on $K_{n-1}$ and to $\chi_n (\phi)$ on $X\setminus K_n$. It is smooth strictly psh 
on a neighbourhood $V_n'$ of
 $K_{n-1}$,
and it is a
psh and continuous exhaustion function for $X$.
It belongs to $\mathcal F$ by Property 
\ref{def-glue} of Definition \ref{defi-admissible-class}.

By the first condition in Definition \ref{defi-admissible-class},
there exist $\epsilon_n$ such that the sequence
$\sum_{n} \epsilon_n u_n \in \mathcal F(X)$. This function satisfies the required properties. 
\end{proof}

\begin{propos}
Let $T$ be an extremal $\mathcal F$-current. All elements of $\mathcal F$
are constant on the support of $T$.
\end{propos}

\begin{proof}
Let $U$ be an admissible neighbourhood of the support of $T$ and
fix $v\in\mathcal F(U)$; we have that $vT$ is again an $\mathcal F$-current.
Indeed, by Lemmas \ref{lemma-defi-limit} and \ref{lemma-intersections-vanish},
we have that the currents
$$\partial v\wedge T, \quad \bar\partial v\wedge T,\quad 
\mbox{ and } \quad \partial\bar\partial v\wedge T$$
are all well-defined and vanish identically. Therefore, $\partial\bar\partial(vT)$
is well defined and vanishes as well. Hence $vT$ is a $\mathcal F$-current.

Now, suppose that $u\in\mathcal F(U)$ is not constant on the support of $T$; then,
without loss of generality, we can suppose  that $\{u<0\}$ and $\{u>0\}$ both intersect
the support of $T$ in a proper subset with non-empty interior. Consider a convex increasing
function $h:\R\to\R$ such that $h(t)=0$ if and only if $t\leq 0$ and $h(t)>0$ otherwise.
Then, by the first part of the proof, 
$h(u)T$ is a $\mathcal F$-current, which is a contradiction with the extremality of $T$.
\end{proof}

\begin{defin}
Let $\mathcal F$ be an admissible class. The \emph{support}
$S^{\mathcal F}=S^{\mathcal F}_X$
of $\mathcal F$ is the union of the supports of all the $\mathcal F$-currents.
\end{defin}

\begin{propos}\label{p:total-support}
Let $F$ be an admissible class. The set $\hat {\mathcal F}$ is closed
for the weak topology of currents. Moreover, there exists
$T \in \hat{\mathcal F}$ such that $\spt T = S^{\mathcal F}$.
\end{propos}

In particular, observe that $S^{\mathcal F}$ is closed in $X$.

\begin{proof}
Let $T_n$ be a sequence of elements in $\hat {\mathcal F}$
and assume that 
there exists a current $T$ on $X$
such that $T_n\to T$ (in the sense of currents). 
We can assume that $T$ is non-zero.
Clearly $T$ is positive and of bidimension (1,1).
Since $i \partial \bar \partial T_n=0$
for all $n$, we deduce that $i \partial \bar \partial T=0$. We need to prove
that $T \wedge i \partial \bar \partial u=0$ for all $u \in \mathcal F$.
We follow the argument given in \cite{Sib} for the case of Levi currents.

Assume first that $u$ is smooth. In this case, for any smooth function
$\chi$ on $X$, we have
$\chi T_n \wedge i \partial \bar \partial u\to  \chi T \wedge i \partial \bar \partial u$.
Since the left hand side of this expression vanishes for all $n$,
we deduce that the same is true for the right hand side.
Since
$\chi$ is arbitrary, we obtain that 
$T \wedge i \partial \bar \partial u=0$, as desired.

Let now $u$ be any element of $\mathcal F$ and let $\chi$ a smooth function with compact support.
It is enough to work locally near the support of $\chi$, and we can assume that this support is arbitrarily small.
Recall that any element in $\mathcal F$ is a a continuous psh function on $X$.
By
 \ref{def-new-dense},
there exists 
a sequence  
$u_n$
 of smooth psh functions, 
 converging to
 $u$ in a 
neighbourhood of $\chi$.
By the arguments above, we have $\chi T \wedge \partial \bar \partial u_n=0$ for all
$n$. The first assertion now follows from Lemma \ref{lemma-defi-limit}.

Let us now prove the second statement. Since $\hat {\mathcal F}$ is closed, it is separable.
Let us consider a countable dense subset $T_j$ of $\hat{\mathcal F}$. For $\epsilon_n$ small enough, consider
the current $T = \sum_n \epsilon_n T_n$ (which is in $\hat {\mathcal F}$) 
and denote by
$S$ its support. By the density of the $T_j$
in $\hat {\mathcal F}$, we obtain that
 any element of $\hat {\mathcal F}$
is supported on $S$.
Hence, $S=S^{\mathcal F}$, and the proof is complete.
\end{proof}

The following proposition gives the
 relation between 
$\mathcal F$-currents and
local maximum sets, and concludes the proof of Theorem \ref{teo-main-singular-currents}.

\begin{propos}\label{p:support-locmax}
Let $\mathcal F$ be an admissible class.
\begin{enumerate} 
\item Assume that $\mathcal F(X)$
contains an exhaustion function. Let $T$ be a $\mathcal F$-current.
If $\spt T$ is compact, then
it is a local maximum set.
\item Assume that $K \subset X$ is an $\mathcal F$-component,
or a compact local maximum set.
 Then there exists
$T \in \hat {\mathcal F}$
such that $\spt T \subseteq K$.
\end{enumerate}
\end{propos}

\begin{proof}
We prove the two assertion separately.
\begin{enumerate}
\item 
Denote $K:= \spt T$, and assume it is not
a local maximum set.
We are going to construct a psh function in neighbourhood
of $K$, which is stricly psh at a point of $\spt T$. This will contradict Lemma \ref{l:stricly-psh-supports}.

In order to construct such a function, we apply Proposition \ref{p:equivalence-locmax}:
there exist $x\in K$, 
$\epsilon>0$,
a neighbourhood $B$ of $x$ (which we can assume to be the unit ball centered at $x=0$ in local coordinates $y$),
and a smooth
strictly psh function $u$ on $B$ such that $u(0)=0$ and
$-\epsilon |y|^2 - \epsilon/8 \leq u(y)\leq -\epsilon |y|^2$ for all $y \in K \cap B$
(by 
\ref{def-new-dense},
we can assume that $u$ is smooth),
 where the first inequality follows
by possibly reducing the ball $B$. 

The function $u$ is only defined near $x$. In order to apply Lemma \ref{l:stricly-psh-supports}, we need to extend it,
as a psh function, on a neighbourhood of $K$. 
By assumption, $X$ admits an exhaustion 
function $\phi\in \mathcal F(X)$. We can also assume that
\[\phi(x)=-\epsilon/4
\quad \mbox{ and } \quad
|\phi| \leq \epsilon/4 
\quad  \mbox{ on } B.\]
By considering a smooth function $\chi$ which is $1$ on a small neighbourhood of $K$, and also compactly
supported in a small neighbourhood of $k$, we can then consider the function
$v$ defined by
\[
v=\begin{cases}
\chi \max_\epsilon(u\phi) & \mbox{ on } B,\\
\chi u & \mbox{ on } X\setminus B.
\end{cases}
\]
where $\max_\epsilon$ is smooth approximation of the $\max$ function.
 As in \cite[Proposition 3.2]{BM}, one can verify that this function is indeed psh in a neighbourhood of $K$,
 and coincides with $u$ in a neighbourhood of $x$. This concludes the proof.

\item
Assume that no $\mathcal F$ current is supported on $K$. Then, by Lemma \ref{lemma:claim-prop-empty}, 
there exists a strictly psh function in $\mathcal F(U)$, where $U$ is an open allowable
neighbourhood of $K$. Fix $x_0 \in K$.
First observe that $du (x_0)\neq 0$.
For $j=1, \dots, 2n-1$, choose a smooth function $\rho_j$ compactly supported in $U$, with the property
that $du_0, d\rho_1, \dots, d\rho_{2n-1}$ are linearly independent at $x_0$.
Property
\ref{def-perturb-t}
in Definition \ref{defi-admissible-class}
gives positive numbers $t_j$ such that $u_0 + t_j \rho_j \in \mathcal F(U)$ for all $j$.
Set $v_j := u_0 + t_j \rho_j$ for all $1\leq j \leq 2n_1$ and $v_0 := u_0$.
As $du_{0}, d\rho_1, \dots, d\rho_{2n-1} $ are linearly independent at $x_0$,
the same holds true
for
$dv_0, dv_1, \dots,  dv_{2n-1}$.
This implies that, in a neighbourhood of $x_0$, we have
$\cap_{j=0}^{2n-1} \{v_j (x)= u(x_0)\}= \{x_0\}$.
Since, by \cite[Corollary 1.11]{Sl_pseudo}, there is a compact 
1-pseudoconcave subset of $K$
where all the $v_j$'s are constant, this gives the desired contradiction.
\end{enumerate}
\end{proof}

In particular,
the following is then a consequence of Propositions
\ref{p:total-support}
and
\ref{p:support-locmax}.

\begin{corol}
Let $\mathcal F$ be an admissible class. 
If the set $S^{\mathcal F}$
is compact, it 
is a local maximum set.
\end{corol}

\section{Localization results}\label{s:final-remarks}

We end this note with some results
about the
 localization of
 $\mathcal{F}$-currents and singular loci to compact (or closed)
 subsets; first of all, by a standard distintegration procedure, we can decompose any $\mathcal{F}$-current
 on the levels of any admissible function, see \cite[Corollary 2.4]{BM} for the case of Levi currents.

\begin{propos}
Let $\mathcal F$ be an admissible class. 
Take $u \in \mathcal F$ and 
let $T$ be a  $\mathcal F$-current.
 There exists a measure $\mu$ on $\R$ and a collection of currents $T_c$, $c\in\R$ such that
\begin{itemize}
\item $T_c$ is supported on $Y_c=\{x\in X\ :\ u(x)=c\}$ for all $c\in\R$;
\item $T_c$ is non zero for $\mu$-almost every $c\in\R$;
\item whenever $T_c\neq 0$, $T_c$ is an $\mathcal F$-current;
\item for every $2$-dimensional form $\alpha$ on $X$ we have
$$\langle T,\alpha\rangle=\int_{\R}\langle T_c,\alpha\rangle d\mu(c)\;.$$
\end{itemize}
Moreover, if $u\in\Ci^1 \cap \mathcal F$ 
and $c$ is a regular value for $u$,
then $T_c=j_*S_c$, where
$j$ is the inclusion of $Y_c$ in $X$ and $S_c$ a current on the real manifold $Y_c$.
\end{propos}

Let now $K$ be a compact set inside the complex manifold $X$ and $\mathcal{U}$ be the set of
all relatively compact neighbourhoods of $K$ in $X$. We can
define the
\emph{singular locus of $\mathcal{F}$ in $K$}
as
$$\Sigma_K^\mathcal{F}=\bigcap_{U\in\mathcal{U}}\Sigma_U^\mathcal{F}\;.$$

\begin{teorema}
Let $\mathcal F$ be an admissible class, $K\subset X$ a compact set, and $\mathcal U$
be defined as above.
\begin{enumerate}
\item The set $\Sigma_K^\mathcal{F}$ can be partitioned in subsets
$\{F_\alpha\}_{\alpha\in A}$ such that, for every $U\in \mathcal{U}$ and
$\phi\in\mathcal{F}(U)$, $\phi$ is constant on $F_\alpha$ for all $\alpha\in A$.
\item Every extremal $\mathcal{F}$-current supported in $K$ is supported in some $F_\alpha$.
\item There exists $T_K\in\hat{\mathcal{F}}$, supported in $K$, such that its support is
maximal, and $\spt T_K\subseteq \Sigma_K^\mathcal{F}$.
\end{enumerate}
\end{teorema}
\begin{proof}

\begin{enumerate}
\item Consider, for each $U\in\mathcal{U}$ and $x\in \Sigma_K^\mathcal{F}$, the
$\mathcal{F}$-component $F_{x,U}$ containing $x$ in $U$; we have that
$\{F_{x,U}\cap K\}_{U\in\mathcal{U}}$ is a net of subsets whose intersection
is a set $F_{\alpha}\subseteq\Sigma_K^\mathcal{F}$ which has the desired property.

\item If $T$ is an extremal $\mathcal{F}$-current supported in $K$, functions in $\mathcal{F}(U)$
are constant on $\spt T$ for all $U\in\mathcal{U}$; 
therefore, we have $\spt T\subseteq F_\alpha$ for some $\alpha$.

\item By Lemma \ref{lemma-14ok}, 
$\mathcal{F}$-currents supported on $K$ are non-zero positive pluriharmonic currents
of bidimension $(1,1)$.
Consider the compact convex set of $(1,1)$-bidimension, positive, $\de\debar$-closed
currents supported in $K$ and with mass $1$.
By Krein-Milman theorem,
 this set 
 is the closure of the convex hull of its extremal elements. Therefore,
 taking a dense sequence of extremal currents $T_j$,
 we can build the current
$$T_K=\sum 2^{-j}T_j,$$
which is again a $\mathcal{F}$-current.
Then, $\spt T_K$ contains the support of every $\mathcal{F}$-current
supported in $K$. By the previous point, we have that $\spt T_K\subseteq\Sigma_K^\mathcal{F}$.
\end{enumerate}
\end{proof}

\begin{rem}\begin{enumerate}
\item The sets $F_\alpha$ as above 
satisfy a local maximum property, outside a suitably defined ``boundary",
i.e., the Hausdorff limit of $\overline{F_{x,U}}\cap bU$ for $U\in\mathcal{U}$.
\item In general, currents $T_j$ may not have disjoint supports: consider $K\simeq\mathbb{P}^2$ as the exceptional divisor in the blow-up of $\C^3$ at the origin; for any allowable class $\mathcal{F}$, $\Sigma_K^\mathcal{F}=K$ and we can pick $T_j$ as the current of integration on some projective line $\mathbb{P}^1$ in $K$. Obviously, all the supports of the $T_j$'s will intersect and in fact the whole $K$ is one unique $\mathcal{F}$-component.
\item The sets $\Sigma_{K}^\mathcal{F}$ and $\spt T_K$ can be different.
Consider in $\mathbb{P}^2$ the set $C:=\overline{K}_+$ for a given H\'enon map $f$
(i.e., a polynomial diffeomorphism of $\mathbb C^2$). Recall that 
$K^+$
 is the set
  of points in $\mathbb C^2$ with bounded forward orbit.
Then $C$
is
compact in the projective plane and, by the main result of \cite{DS_rigid}, it 
supports only one positive, $\de\debar$-closed, $(1,1)$-current of mass
$1$, which is the Green current $T_+$. So,
$T_+=T_{C}$ and $\spt T_{C}$ is
equal to the closure in $\mathbb P^2$
of $J_+=\partial K_+$. 
If $f$ is appropriately chosen, 
some connected component
$\Omega$
of  $K^+ \setminus J_+$ is a Fatou-Bieberbach domain,
biholomorphic to $\C^2$, and we have $\partial \Omega = J^+$
\cite{RR,BS2}.
Therefore, every psh function
which is continuous on $C':= \Omega \cup J^+$
needs to
 be constant
on $C'$
 (since it restricts 
 to a bounded psh function on $\Omega\cong \C^2$).
  Hence, $\Sigma_{C'}^\mathcal{F}=C' \supsetneq \spt T_{C'}= \spt T_C$.
  (We used in this remark the notations $C$ and $C'$ to avoid confusion with
  the set $K$ already defined in the dynamical setting
  as the set of points with bounded both forward and backward orbits).
\item In the case of all continuous plurisubharmonic functions, the set $\Sigma_K^\mathcal{F}$
is the same as the psh kernel defined in \cite{MT}, or the core of a compact set as defined in \cite{Sh21}.
\end{enumerate}
\end{rem}

\end{document}